\documentclass[12pt]{amsart}
\usepackage{amsmath}

\usepackage{amsmath,amssymb,dsfont}
\usepackage{enumerate}
\usepackage{mathrsfs} 
\usepackage{txfonts} 
\usepackage{graphicx}
\usepackage[british]{babel} 
\usepackage{arydshln}	
\usepackage[round]{natbib}

\newtheorem{lemma}{Lemma}

\newtheorem{theorem}[lemma]{Theorem}
\newtheorem{proposition}[lemma]{Proposition}

\newtheorem{example}[lemma]{Example}

\newtheorem{remark}[lemma]{Remark}

\newtheorem{assumption}[lemma]{Assumption}

\newcommand{\Cc}{\mathscr{C}}
\newcommand{\dsD}{\mathds{D}}
\newcommand{\Ee}{\mathscr{E}} 
\newcommand{\N}{\mathds{N}}  
\newcommand{\R}{\mathds{R}}   
\newcommand{\Ss}{\mathscr{S}} 

\newcommand{\be}{\begin{equation}}
\newcommand{\ee}{\end{equation}}
\newcommand{\bea}{\begin{eqnarray*}}
\newcommand{\eea}{\end{eqnarray*}}
\newcommand{\ben}{\begin{enumerate}[a)]} 
\newcommand{\bEN}{\begin{enumerate}[1.]}
\newcommand{\bECR}{\begin{enumerate}[(I)]} 
\newcommand{\bESR}{\begin{enumerate}[(i)]} 
\newcommand{\een}{\end{enumerate}}
\newcommand{\eEN}{\end{enumerate}}
\newcommand{\var}{\mathrm{var}}

\DeclareMathOperator*{\vectorize}{vec}

\DeclareMathOperator*{\argmin}{arg\,min}
\DeclareMathOperator*{\argmax}{arg\,max}
\DeclareMathOperator*{\ave}{ave}
\DeclareMathOperator*{\trace}{tr}
\renewcommand*{\vec}{\vectorize\!}
\DeclareMathOperator*{\mat}{mat}
\def\matpp{\mat\nolimits_{p\times p}}
\newcommand{\X}{\mathbb{X}}
\newcommand{\Rpp}{\R^{p \times p}}
\def\Sp{\Ss^+_p}

\def\hK{\hat{K}}
\def\hk{\hat{k}}

\def\hS{\hat{S}\!}

\def\hV{\hat{V}}

\def\hmu{\hat{\mu}}

\def\hSigma{\hat{\Sigma}}
\def\tSp{\tilde{\Ss}^+_p}

\def\tQKG{\tilde{Q}_{K(G)}}
\def\tQDG{\tilde{Q}_{D(G)}}
\def\tQK{\tilde{Q}_K}
\def\tQD{\tilde{Q}_D}
\def\tPG{\tilde{P}_G}
\def\mle{\mbox{\scriptsize mle}}

\renewcommand\arraystretch{1.05}

\begin{document}
\selectlanguage{british}
\title[Non-decomposable elliptical graphical models]{Robust estimators for non-decomposable elliptical graphical models}
\author{D. Vogel and D. E. Tyler}

\address{
	Fakult\"at f\"ur Mathematik, Ruhr-Universit\"at Bo\-chum, 44780 Bochum, Germany
}
\email{vogeldts@rub.de}

\address{
	Department of Statistics, 
	Rutgers University, 
	Piscataway, NJ 08854, USA}

\email{dtyler@rci.rutgers.edu}

\thanks{Research supported in part by the 
Collaborative Research Grant 823, Project C3 {\em Analysis of Structural 
Change in Dynamic Processes}, of the German Research 
Foundation. }
\begin{abstract}
Asymptotic properties of scatter estimators for elliptical graphical models are studied. Such models impose a given pattern of zeros on the inverse of the shape matrix of an elliptically distributed random vector. In particular, we introduce the class of graphical $M$-estimators and compare them to plug-in $M$-estimators. It turns out that, under suitable conditions, both approaches yield the same asymptotic efficiency. Furthermore, the results of this paper apply to both decomposable and non-decomposable graphical models and so generalize the results for decomposable models given by \citet{Vogel2011} for the plug-in $M$-estimators.
\end{abstract}

\keywords{
affine equivariance; delta method; deviance test; Gaussian graphical model; $M$-estimator; partial correlation.
}

\maketitle
\tableofcontents

\section{Introduction \& motivation: non-decomposable covariance selection models}

The research presented in this article originates from the authors' interest in robustifying and generalizing classical Gaussian graphical modelling. We outline the idea.

Suppose we observe realizations of a $p$-dimensional random vector $X = (X_1,\ldots,X_p)$ with non-singular covariance matrix $\Sigma$. Its inverse $K = \Sigma^{-1}$ is called the concentration matrix. A zero entry in $K$ at position $(i,j)$ for $i,j = 1,\ldots,p$, $i \neq j$,  means that $X_i$ and $X_j$ are partially uncorrelated given all other components of $X$. This means, if we denote by $(\hat{X}_i, \hat{X}_j)$ the orthogonal projections of $(X_i, X_j)$ onto the space of all affine linear functions of the other components of $X$, the residuals $X_i - \hat{X}_i$ and $X_j - \hat{X}_j$ are uncorrelated.     

When studying more than two variables jointly, the partial correlations among each two of them are arguably more informative than the marginal correlations, because they allow to assess to what degree the dependence between two variables is explained by their joint dependence on other variables. In fact, considering only marginal correlation may lead to wrong conclusion, which is nicely exemplified by Simpson's paradox \citep[e.g.][Chapter 1.4]{Edwards2000}. We are therefore interested in the statistical task of determining the zero entries of $K$.

We define the partial correlation graph $G = (V,E)$ of $X$ by setting 
$V = \{1,\ldots,p\}$ and 
$E = \{ \, \{i,j\} \mid  i,j =  1,\ldots,p,\ j < i, \ K_{i,j} \neq 0 \}$, where $K = (K_{i,j})_{i,j=1,\ldots,p}$. Thus, the nodes $i$ and $j$ are connected in $G$ by an undirected edge if and only if $X_i$ and $X_j$ are partially correlated given all other variables. The task of determining the zero-entries of $K$ can be rephrased to find the partial correlation graph of the data.

Let $\Ss_p$ and $\Sp$ denote the set of all symmetric $p \times p$ matrices and the set of all positive definite $p \times p$ matrices, respectively.
For any graph $G = (V,E)$ let further $\Ss^+_p(G)$ be the set of matrices $A \in \Sp$ with zero entries at off-diagonal positions specified by $G$, i.e., $A_{i,j} = 0$ 
for all $i,j =  1,\ldots,p$, $j \neq i$, with $\{i,j\} \notin E$. We call any set of $p$-dimensional probability measures with the common property that they possess a concentration matrix $K \in \Ss^+_p(G)$ a covariance selection model induced by $G$. We call a covariance selection model consisting of all regular, i.e.\ with full rank covariance matrix, $p$-variate Gaussian distributions a Gaussian graphical model and denote it by $N_p(G)$, i.e., 
$N_p(G) = \{ N_p(\mu,K^{-1}) \mid \mu \in \R^p, K \in \Ss^+_p(G) \}$.

For a Gaussian vector $X = (X_1,\ldots,X_p)$ the partial uncorrelatedness of $X_i$ and $X_j$, $i,j = 1, \ldots, p$, $j \neq i$, is equivalent to their conditional independence given the other components of $X$. Usually, the terms covariance selection model and Gaussian graphical model are used synonymously for families of Gaussian distributions. We prefer to distinct between both, since we focus on the analysis of second moments and will also study covariance selection models for non-Gaussian distributions. We continue by reviewing some aspects of the statistical modelling of Gaussian graphical models.

The parametric family $N_p(G)$ is a regular exponential model parametrized by $\mu$ and $K$, in total $p(p+3)/2 - q$ parameters, where $q$ is the number of absent edges in $G$, and the maximum likelihood paradigm offers a way of efficient estimation and testing.

\medskip
{\sc Maximum likelihood estimator.}
Based on independent and identically distributed observations $X_1,\ldots,X_n$ stemming from $N_p(G)$ for some graph $G = (V,E)$, the maximum likelihood estimator $\hSigma_G$ of $\Sigma$ in the model $N_p(G)$ is defined for $n > p+1$ as the solution of 
\be \label{eq:h_G}
\begin{cases}
 \  (\hSigma_G)_{i,j} = (\hSigma_n)_{i,j},  		&	\qquad    \{i,j\} \in E \  \vee  \ i = j, \\
 \  (\hSigma_G^{-1})_{i,j} = 0,    							& \qquad    \{i,j\} \notin E,  \ i \neq j, \\
\end{cases}
\ee  
where $\hSigma_n$ is the sample covariance matrix computed from $X_1,\dots,X_n$. A unique and positive definite solution $\hSigma_G$ of (\ref{eq:h_G}) exists for any positive definite $\hSigma_n$, see also \citet{Grone1984}. Furthermore there are algorithms that have been shown to converge to the right solution. The general likelihood theory for exponential families yields that $\hSigma_G$ is asymptotically normal for any $G$.

\medskip
{\sc Likelihood ratio test.}
Consider two nested graphs $G_0 = (V,E_0)$ and $G = (V,E)$ with $V = \{1,\ldots,p\}$ and $E_0 \subsetneqq E$. The likelihood ratio for the hypothesis $K \in \Ss^+_p(G_0)$ in the model $N_p(G)$ is 
$L_n(G_0,G) = ( \det\hSigma_{G} / \det\hSigma_{G_0} )^{n/2}$.
Under the null hypothesis $K \in \Sp(G_0)$, the related deviance test statistic
\[
	D_n(G_0,G) = - 2 \log L_n(G_0,G) = n \left(\log\det\hSigma_{G_0} - \log\det\hSigma_{G}\right)
\]
converges for $n \to \infty$ in distribution to a $\chi^2$ distribution with $q_0-q$ degrees of freedom, where $q_0$ and $q$ are the numbers of absent edges in $G_0$ and $G$, respectively.

\medskip
{\sc Model search.}
Many classical model selection procedures consist of a repeated application of the deviance test.
For instance, a simple model search, known as backward elimination, starts with the saturated model and, in each step, removes one edge. The deviances between the current model and all models with exactly one edge less are computed. The edge with the smallest deviance difference is deleted, unless all edges are significant. 

\medskip
A serious drawback of this likelihood approach, which was originated by \citet{Dempster1972} and is treated in detail in \citet{Lauritzen1996}, is the lack of robustness, and alternatives have been proposed. \citet{Vogel2011} study estimators of the type $\hS_G = h_G(\hS_n)$ within the class of elliptical distributions, where  
\[
	h_G:\Ss^+_p \to \Ss^+_p
\]
denotes the function that maps $\hSigma_n$ to $\hSigma_G$, cf.~(\ref{eq:h_G}), and $\hS_n$ can be any affine equivariant and asymptotically normal scatter estimator. See Assumption \ref{ass:1} for a precise statement of these terms. In this more general setting, the asymptotic normality of $h_G(\hS_n)$ and the convergence of 
\be \label{eq:gen.dev}
	D_n(G_0,G_1,\hS_n) = n \{ \log h_{G_0}(\hS_n) - \log h_{G_1}(\hS_n) \}
\ee 
under $G_0$ can not be deduced from general likelihood results. \citet{Vogel2011} give proofs for decomposable models.
An undirected graph $G = (V,E)$ and any corresponding covariance selection model is called decomposable or chordal or triangulated, if every cycle of length greater than 3 possesses a chord. For such graphs $G$, the function $h_G$ has an explicit form, from which its derivative can be computed. By means of the delta method one can derive the asymptotic normality of $h_G(\hS_n)$, and subsequently the $\chi^2$ limit of $D_n(G_0,G_1,\hS_n)$. 

One main objective of this paper is to extend this approach to non-decomposable models. We will give an explicit expression for the asymptotic covariance matrix of the plug-in estimator $h_G(\hS_n)$. We further introduce an alternative class of scatter estimators under the covariance selection model $G$, which we call graphical $M$-estimates. We show that the graphical $M$-estimator is asymptotically equivalent to the plug-in estimator $h_G(\hS_n)$ if $\hS_n$ is the corresponding unrestricted $M$-estimate.

\section{Main result}
\label{sec:main}

In this section we give the derivative of the function $h_G$. Towards this end, we have to introduce some notation.
The Kronecker product $A \otimes B$ of two matrices $A,B \in \R^{p \times p}$ is defined as the $p^2 \times p^2$ matrix with entry $a_{i,j} b_{k,l}$ at position $((i-1)p + k, (j-1)p + l)$. Let $\vec A$ be the $p^2$-vector obtained by stacking the columns of $A \in \R^{p \times p}$ from left to right underneath each other and $\matpp: \R^{p^2} \to \R^{p \times p}$ denote the inverse operator to $\vec$\, for $p \times p$ matrices.
Letting $e_1, \ldots, e_p$ be the unit vectors in $\R^p$, we further define the matrices
\[
	  K_p = \sum\nolimits_{i=1}^p \sum\nolimits_{j=1}^p e_i^{} e_j^T \otimes e_j^{} e_i^T,
	  \qquad
  M_p = \frac{1}{2}\left( I_{p^2} + K_p \right),
\]
where $I_{p^2}$ denotes the $p^2 \times p^2$ identity matrix. The matrix $K_p$ is orthogonal and is commonly referred to as the commutation matrix. It can also be viewed as the transpose operator since $K_p \vec A = \vec A^T$. We call the idempotent matrix $M_p$ 
the symmetrization matrix since it maps $\vec A$ to $\frac{1}{2} \vec(A + A^T)$. 
Further, let $m = p(p+1)/2$ and, for any matrix $A \in \Ss_p$, let $v(A)$ be the $m$-vector that is obtained by deleting the super-diagonal elements of $A$ from $\vec A$. The duplication matrix
$D_p \in \R^{p^2 \times m}$ is the matrix that maps $v(A)$ to $\vec A$. It has exactly one 1-entry in each row and is zero otherwise. Its Moore-Penrose inverse $D^+_p = (D_p^T D_p)^{-1}D_p^T$ then reduces $\vec A$ to $v(A)$ for any symmetric matrix $A \in \Rpp$. We have the following identities:
\[
   D_p D_p^+ = M_p, \qquad D_p^+ D_p = I_m \quad \mbox{and} \quad M_p (A \otimes A) M_p = M_p (A \otimes A) = (A \otimes A) M_p
\]
for any $A \in \Rpp$. More on these concepts and their properties can be found in \citet{Magnus1999}.
On the set $\Pi_p = \{ (i,j) \mid i,j = 1,\ldots, p\}$ of the positions of a $p \times p$ matrix we declare a strict ordering $\prec_p$ by  
\[
	(i,j) \prec_p (k,l)\quad  \mbox{if} \quad \ (j-1)p + i \le (l-1)p + k  \quad\mbox{for } \ (i,j), (k,l)  \in \Pi_p.  
\]
This corresponds to the ordering imposed by the operation $\vec A$ on the components of $A$.
For any subset $Z = \{ z_1,\ldots,z_r \} \subset \Pi_p$, where $z_k = (i_k, j_k)$ ($k = 1,\ldots,r$) and $z_1 \prec_p \ldots \prec_p z_r$, define the matrix $Q_Z \in \R^{r \times p^2}$ as follows: each line consists of exactly one entry 1 and zeros otherwise. The $1$-entry in line $k$ is in column $(j_k - 1)p + i_k$. Thus $Q_Z\!\vec A$ contains those elements of $A$ that are specified by $Z$ in the order they appear in $\vec A$. 

For a graph $G = (V, E)$ with $V = \{ 1,\ldots,p\}$ we define the following subsets of $\Pi_p$,
\[
	D(G) = \{\, (i,j) \mid i,j = 1,\ldots,p,\ j < i,\ \{i,j\} \notin E\, \}, 
\]\[
	K(G) = 
	\{\, (i,j)\mid i,j = 1,\ldots,p,\ j < i,\ \{i,j\} \in E \,\} \ \cup \ 
	\{\, (i,i) \mid i = 1,\ldots,p \,\}.
\]
Thus, $D(G)$ gathers all sub-diagonal zero-positions that $G$ enforces on a concentration matrix, and $K(G)$ collects all diagonal positions and all sub-diagonal edge positions. The sets $D(G)$ and $K(G)$ contain $q$ and $m-q$ elements, respectively, where $q$ is the number of absent edges in $G$. We write $Q_D$ and $Q_K$ short for $Q_{D(G)}$ and $Q_{K(G)}$, respectively. Note that $Q_{D(G) \cup K(G)} \vec A = D_p^+ \vec A = v(A)$ for any $A \in \Ss_p$.
Finally, let $\tQD = Q_D D_p$ and $\tQK = Q_K D_p$. We are now ready to formulate our main result.
\begin{proposition} \mbox{{}\\{}}
\label{prop:1}
\bECR
\item \label{th:1.1}
The function $h_G$ is continuously differentiable on $\Sp$.
\item \label{th:1.2} 
The derivative of $h_G$ at $A \in \Sp$ is
\be \label{eq:derivative1}
\dsD h_G(A) \ = \ 
 M_p \ - \  
 M_p Q_D^T \left\{ Q_D M_p (A_G^{-1} \otimes A_G^{-1}) Q_D^T \right\}^{-1} Q_D (A_G^{-1} \otimes A_G^{-1}) M_p ,
\ee
where $A_G$ denotes $h_G(A)$.
\een
\end{proposition}
\begin{theorem} \label{th:1}
Let $(\hV_n)_{n\in\N}$ be a sequence of random $p\times p$ matrices such that $\sqrt{n} \vec\, (\hV_n-V)$ converges in distribution to a $p^2$-valued random vector $Z$ for some fixed matrix $V \in \Sp$. 
\bECR
\item \label{th:1(I)}
Then $\sqrt{n} \vec\, \{ h_G(\hV_n) - h_G(V) \} \to \dsD h_G(V) Z$  in distribution.
\item \label{th:1(II)}
If additionally $Z$ is normal with mean zero and covariance matrix  
\be \label{eq:W_V}
	W_V = 2 \sigma_1 M_p (V \otimes V)\ + \ \sigma_2 \vec V (\vec V)^T
\ee
for some scalars $\sigma_1 \ge 0$ and $\sigma_2 \ge - 2\sigma_1/p$, then $\dsD h_G(V) Z$ is $p^2$-variate normal with mean zero and covariance matrix
\be \label{eq:W_{V,G}1}
	W_{V,G} = 2 \sigma_1  \dsD h_G(V) \left( V \otimes V \right) \left\{ \dsD h_G(V) \right\}^T
	\ + \ \sigma_2 \vec V_G (\vec V_G)^T,
\ee
where $V_G$ denotes $h_G(V)$.
\item \label{th:1(III)}
If the assumptions of part (\ref{th:1(II)}) hold and $V^{-1} \in \Sp(G)$ , i.e., $V = h_G(V)$, then $W_{V,G}$ reduces to
\be \label{eq:W_{V,G}2}
	W_{V,G} = 2 \sigma_1 M_p 
	\left[ 
			V\!\otimes\! V - Q_D^T 
			\left\{ Q_D M_p (V^{-1}\!\otimes\! V^{-1}) Q_D^T \right\}^{-1} 
			Q_D  M_p
	\right]
	\ + \ \sigma_2 \vec V (\vec V)^T.
\ee
\item \label{th:1(IV)}
Letting $u = Q_K\! \vec\,(V^{-1})$ and $\hat{u}_G = Q_K\!\vec\,\{ h_G(V_n)^{-1}\}$, we have under the assumptions of part (\ref{th:1(III)}) that 
\[
	\sqrt{n}(\hat{u}_G - u) \to N_{m-q}\left(\, 0,\, W_{u,G} \right)
\]
in distribution with 
\be \label{eq:W_{u,G}}
	W_{u,G} = 2\sigma_1 \left\{ \tQK D_p^T (V \otimes V) D_p \tQK^T \right\}^{-1} + \sigma_2 u u^T.
\ee
\end{enumerate} 
\end{theorem}

\begin{remark} \mbox{{}\\{}}
\bECR
\item
The assumption (\ref{eq:W_V}) on the covariance matrix  of $\vec Z$ in Theorem \ref{th:1} (\ref{th:1(II)}) may appear somewhat arbitrary. In fact, it is equivalent to require, along with normality, that $Z = (T \otimes T) Z$ in distribution for any matrix $T \in \Rpp$ such that $V^{-1/2} T V^{1/2}$ is orthogonal \citep[see also][Corollary 1]{Tyler1982}. This asymptotic invariance property is often encountered when studying the distribution of scatter estimators. It holds, for example, for affine equivariant scatter estimators at elliptical distributions, cf.~Lemma \ref{lem:ae}.
\item
The usual application of Theorem \ref{th:1} will be that $\hV_n$ is a scatter estimator of the unknown scatter matrix $V$. If $V^{-1} \in \Sp(G)$, then $u$ is simply the relevant part of $V^{-1}$ with all zeros and symmetry redundancies removed. In particular, $W_{u,G}$ is a full-rank matrix.
\een
\end{remark}

\section{Affine equivariant scatter estimators at elliptical distributions}
\label{sec:aese}
We describe a general situation where Theorem \ref{th:1} 
applies. Consider the class $\Ee_p$ of all $p$-dimensional, continuous, elliptical distributions, i.e., distributions possessing a $p$-dimensional Lebesgue density $f$ of the form
\begin{equation} \label{eq:density}
	f(x) = \det(S)^{-\frac{1}{2}} g\big\{(x-\mu)^T S^{-1} (x-\mu)\big\}
\end{equation}
for some $\mu \in \mathds{R}^p$, $S \in \Sp$ and $g:[0,\infty) \to [0,\infty)$ such that $f$ integrates to 1. 
Let $E_p(\mu,S,g)$ denote the distribution described by (\ref{eq:density}). 
Note that it is not necessary to assume in general that the elliptical distribution possesses second moments or even first moments.
For a random sample  $X_1,\ldots,X_n$ let $\X_n = (X_1,\ldots,X_n)^T$ denote the $n \times p$ data matrix. Let further $\hS_n$ be an $\Ss_p$-valued scatter estimator satisfying Assumptions \ref{ass:1} and \ref{ass:2} below.
\begin{assumption}[Affine equivariance]
\label{ass:1}
There is a continuously differentiable function $\xi: \Ss_p \to [0,\infty)$ with $\xi(I_p) = 1$ such that
\[
	\hS_n(\X_n A^T + 1_n b^T ) = \xi(AA^T) A \hS_n(\X_n) A^T
\]	
for any $b \in \R^p$ and full rank $A \in \R^{p \times p}$, where $1_n$ is the $n$-vector consisting of ones. 
\end{assumption}
%
%
%
%
%
This is a generalization of the strict affine equivariance for scatter estimators, which corresponds to $\xi \equiv 1$. We use this weaker condition since we want to include shape estimators that give no information about the overall scale. They are usually scaled to $\det \hS_n = 1$ and do hence not satisfy strict affine equivariance. An example is the distribution-free $M$-estimator by \citet{Tyler1987}.
\begin{assumption}[Asymptotic normality] \label{ass:2}
The random vectors $X_1,\ldots,X_n$ are independent and identically $E(\mu,S,g)$ distributed, and there is a matrix $V \in \Sp$ such that $\sqrt{n} \vec\, \{ \hS_n(\X_n) - V \}$ converges in distribution to a $p^2$-variate, centered normal variable $Z$.
\end{assumption}
\begin{lemma} \label{lem:ae}
Under Assumptions \ref{ass:1} and \ref{ass:2} we have
\bECR
\item \label{lem:ae(I)}
$V = \eta S$ for some $\eta \ge 0$ and
\item \label{lem:ae(II)}
$Z$ satisfies the assumption of Theorem \ref{th:1} (\ref{th:1(II)}), i.e., it has covariance matrix $W_V$. 
\een
\end{lemma}
The class of scatter estimators satisfying Assumptions \ref{ass:1} and \ref{ass:2} is large. One important motivation for considering alternatives to the sample covariance matrix is the lack of robustness of the latter. Over the last decades, the robustness literature has produced many proposals of affine equivariant, robust estimators. Prominent examples of such estimators are $M$-estimators \citep[e.g.][]{Maronna1976}, Stahel-Donoho estimators, S-estimators \citep[e.g.][]{Davies1987}, C$M$-estimator \citep[][]{Kent1996}, Oja sign and rank matrices \citep{Ollila2003,Ollila2004}. See, e.g., the overview article by \citet{Zuo2006} or the book by \citet*{Maronna2006} for further reading. Having outlined the general situation, we want to take a look at three specific examples. 
\begin{example}[Sample covariance matrix]
The sample covariance matrix $\hSigma_n$ fulfils Assumption \ref{ass:1} and, if the fourth moments of $E_p(\mu,S,g)$ are finite, i.e., if $\int_{\R^p} ||x||^4 g(||x||^2) dx < \infty$, then $\hat{\Sigma}_n(\X_n)$ fulfils also Assumption \ref{ass:2}. Hence by Lemma \ref{lem:ae}, the conditions of Theorem \ref{th:1} (\ref{th:1(II)}) are met. The scalars $\sigma_1$ and $\sigma_2$ are identified as
$\sigma_1 = 1 + \kappa/3$ and $\sigma_2 = \kappa/3$, where $\kappa$ denotes the excess kurtosis of any component of $X \sim E_p(\mu,S,g)$. 
Assuming further that the data is normal, i.e.\ that $g(y) = (2\pi)^{-p/2} \exp(-y/2)$, $y \ge 0$, then $\kappa = 0$ and $S = \var(X) = V$, i.e. the scalar $\eta$ in Lemma \ref{lem:ae} equals 1. If we let $k = Q_K \vec\, (\Sigma^{-1})$ and $\hat{k}_G = Q_K \vec\, [\{h_G(\hSigma_n)\}^{-1}]$, we have in particular by part (\ref{th:1(IV)}) of Theorem \ref{th:1} that
\[
	\sqrt{n} (\hat{k}_G - k) \to N_{m-q}\left(\ 0\, ,\ 2 \left\{ \tQK D_p^T (\Sigma \otimes \Sigma) D_p \tQK^T \right\}^{-1} \right)
\]
in distribution. This result is also given in a much different notation in \citet[][Section 5.3]{Roverato1998}.
\end{example}

\begin{example}[Elliptical maximum likelihood estimator] \label{ex:emle}
Consider a fixed function $g$ and the maximum likelihood estimator $(\hat{\mu}_g, \hS_g)$ of $(\mu, S)$ in the elliptical family
\[
	\Ee_p(g) = \left\{ E_p(\mu,S,g)\, \middle|\,  \mu \in \R^p, S \in \Sp \right\}. 
\]
Letting $\hS_g = \hat{K}_g^{-1}$, the maximum likelihood estimator is  the solution to the maximization problem 
\be \label{eq:emle}
	(\hat{\mu}_g, \hat{K}_g) = \argmax_{\mu \in \R^p, K \in \Sp} \left[ n \log \det K + 2 \sum\nolimits_{i=1}^n \log g\left\{ (X_i - \mu)^T K (X_i - \mu)  \right\} \right].
\ee
For results on the existence and uniqueness of the solution see, e.g., \citet{kent:tyler:1991}. Any solution to (\ref{eq:emle}) fulfils Assumption \ref{ass:1}. Under the usual regularity conditions on the density \citep[][pp.~429--430]{lehmann:1983}, we have that, if the data $X_1,\ldots,X_n$ stem from the distribution $E_p(\mu,S,g)$, the elliptical maximum likelihood estimator $\hS_g$ fulfils also Assumption~\ref{ass:2} and, by Lemma \ref{lem:ae}, the conditions of Theorem \ref{th:1} (\ref{th:1(II)}). The scalars are $\eta = 1$, 
\[
		\sigma_1  =  \frac{p(p+2)}{E\left\{R^2 u^2(R)\right\}},  \qquad 
	\sigma_2  =  - \frac{2\sigma_1(1-\sigma_1)}{2 + p(1-\sigma_1)},
\] 
where $R = (X - \mu)^T S^{-1} (X-\mu)$ for $X \sim E_p(\mu,S,g)$ and $u(y) = -2 g'(y)/g(y)$, $y \ge 0$, see \citet{Tyler1982}. 
\end{example}

\begin{example}[Multivariate $M$-estimators] \label{ex:M}
The $M$-estimators of multivariate location and scatter $(\hat{\mu}_n, \hS_n)$ are generalizations of the maximum likelihood estimators obtained by replacing $-2 \log g$ in (\ref{eq:emle}) with an arbitrary function $\rho$. An $M$-estimator can then be expressed as the solution to the minimization problem
\be \label{eq:M}
	(\hat{\mu}_n, \hS_n) = \argmin_{\mu \in \R^p, \Sigma \in \Sp} \left[ \sum\nolimits_{i=1}^n \rho\left\{ (X_i - \mu)^T \Sigma^{-1} (X_i - \mu)  \right\} + n \log \det \Sigma \right].
\ee
A more general definition for the $M$-estimates of multivariate location and scatter is given as any solution to the following simultaneous  
$M$-estimating equations 
\be \label{eq:Mest}
\begin{cases}
	\displaystyle
 \  0 = \sum\nolimits_{i=1}^n  u_1(\hat{R}_i)(X_i - \hmu_n) ,  \\[8pt] 
    \displaystyle
 \  \hS_n = n^{-1} \sum\nolimits_{i=1}^n u_2(\hat{R}_i) (X_i-\hmu_n)(X_i-\hmu_n)^T, \\
\end{cases}
\ee
where $\hat{R}_i = (X_i-\hmu_n)^T \hS_n^{-1}(X_i-\hmu_n)$, for some functions $u_1$ and $u_2$, see \citet{Maronna1976} or \citet{Huber2009}.  
For the special case $u_1 = u_2 = u$, where $u(s) = \rho'(s)$, equations (\ref{eq:Mest}) yield the critical points of (\ref{eq:M}). 
Any solution to (\ref{eq:Mest}) fulfils Assumption \ref{ass:1}. Under general regularity conditions \cite[]{Maronna1976} the multivariate
$M$-estimators are asymptotically normal. Also, if the data represent a random sample from the distribution $E_p(\mu,S,g)$, then the $M$-estimators of scatter satisfy Assumption~\ref{ass:2} and hence the conditions of Theorem \ref{th:1} (\ref{th:1(II)}). The scalars are 
\[
		\sigma_1  =  \frac{(p+2)^2 \gamma_1}{(2\gamma_2 + p)^2},  \qquad 
	\sigma_2  =  \gamma_2^{-1}\left\{(\gamma_1-1)-\frac{2\gamma_1(\gamma_2 -1)(p+\{p+4\}\gamma_2)}{(2\gamma_2 + p)^2}\right\},
\] 
where $\gamma_1 = E[\phi_2^2(\eta R)]/\{p(p+2)\}$ and $\gamma_2 = E[\eta R \phi_2'(\eta R)]/p$, with $\phi_2(s) = su_2(s)$ and $\eta$ being the solution to $E[\phi_2(\eta R)] = p$, see \citet{Tyler1982}. 
\end{example}

\section{Graphical $M$-estimates} 
\label{sec:Mgest}
Pursuing Example \ref{ex:emle} above a little further, we call, for a given graph $G$ and a fixed function $g$, 
\[
	\Ee_p(g,G) = \left\{ \ E_p(\mu,S,g)\  \middle| \  
		\mu \in \R^p,\  S^{-1} \in \Sp(G) \,
	 \right\}
\]
the elliptical graphical model induced by $g$ and $G$. 
Within this model, the estimator $\hS_{g,P} = h_G(\hS_g)$ provides a sensible estimate for $S$\!, where $\hS_g$ is the elliptical maximum likelihood estimator introduced in Example \ref{ex:emle}. We call $\hS_{g,P}$ the plug-in maximum likelihood estimator. An alternative is the actual maximum likelihood estimator of $S$ in the model $\Ee_p(g,G)$. Define $\hS_{g,\mle} = \hK_{g,\mle}^{-1}$ and $(\hmu_{g,\mle}, \hK_{g,\mle})$ as the solution of 
\be \label{eq:gemle}
	(\hmu_{g,\mle}, \hK_{g,\mle}) = \argmax_{\mu \in \R^p, K \in \Sp(G)} 
	\left[ 
		n \log \det K 
		+ 2 \sum\nolimits_{i=1}^n \log g\left\{ (X_i - \mu)^T K (X_i - \mu)  \right\} 
	\right].
\ee
We call this estimator the graphical maximum likelihood estimator. It is of interest to compare the estimators $\hS_{g,\mle}$ and $\hS_{g,P}$. The maximum likelihood estimator proves to be most efficient in many situations. 
The function $h_G$ was derived from considerations for maximum likelihood estimation in Gaussian graphical models. Thus, we expect the plug-in maximum likelihood estimator to be less efficient than the proper, graphical maximum likelihood estimator in a non-normal elliptical graphical model. 
We show in the following though that the suspected loss in efficiency is nil asymptotically. We treat this question within the more general framework of $M$-estimators.

In the following, let $\hS_n$ denote an $M$-estimator of scatter, i.e.\ $\hS_n$ is the scatter part of the solution $(\hmu_n,\hS_n)$ of the simultaneous $M$-estimating equations (\ref{eq:Mest}). Suppressing the dependence on $G$, we call $(\hmu_P,\hS_P) = \left(\hmu_n,h_G(\hS_n)\right)$ the plug-in $M$-estimators of location and scatter under $G$. Also, analogously to the graphical maximum likelihood estimators we introduce the graphical $M$-estimators of multivariate location and scatter under $G$, denoted $(\hmu_M, \hS_M)$, as a solution to
 \be \label{eq:Mgest}
	(\hmu_{M}, \hK_{M}) = \argmax_{\mu \in \R^p, K \in \Sp(G)} 
	\left[ 
	   n \log \det K \, -\,  \sum\nolimits_{i=1}^n \rho\left\{ (X_i - \mu)^T K (X_i - \mu)\right\}  
	\right],
\ee
where $\hS_M = \hK_M^{-1}$, 
or more generally as a solution to the $M$-estimating equations
\begin{equation} \label{eq:Mgest1}
\begin{cases}
	\displaystyle
 \  0 = \sum\nolimits_{i=1}^n  u_1(\hat{R}_{i,M})(X_i - \hmu_M) ,  \\[8pt] 
  \displaystyle
 \  (\hS_M)_{j,k} =  e_j^T \left\{ n ^{-1}\sum\nolimits_{i=1}^n u_2(\hat{R}_{i,M}) (X_i-\hmu_M)(X_i-\hmu_M)^T \right\} e_k,
 			&	\quad    \{j,k\} \in E \  \vee  \ i = j, \\[10pt]
 	\displaystyle		
 \  (\hS_M^{-1})_{j,k}  = 0,     
 			& \quad    \{j,k\} \notin E,  \ i \neq j,
 \end{cases}
\ee
where $\hat{R}_{i,M} = (X_i-\hmu_M)^T \hS_M^{-1}(X_i-\hmu_M)$.
The special case $u_1 = u_2 = \rho'$ corresponds to the critical points of (\ref{eq:Mgest}). The proof of this last statement is given in the appendix. It is worth noting that, in general, knowing  $A_{j,k}$ for 
$(j,k) \in K(G)$ and  $(A^{-1})_{j,k}$  for $(j,k) \in D(G)$ uniquely determines the symmetric positive definite matrix $A$, see e.g.\ Theorem 1 in \citet{Speed1986}. Thus (\ref{eq:Mgest1}) consists of $p(p+3)/2-q$ equations to be solved for the same number of unknowns. This may be more clearly visible when we write 
$\hS_M$ as 
\[
	\hS_M = \hK_M^{-1}, \qquad  \hK_M = \matpp\left( D_p \tQK^T \hk_M \right),
\] 
where $\hk_M$ is a vector of length $m-q$ and $(\hmu_M,\hk_M)$ the solution of 
\be \label{eq:Mgest2}
\begin{cases}
	\displaystyle
 \  0 = \sum\nolimits_{i=1}^n  u_1(\hat{R}_{i,M})(X_i - \hmu_M) ,  \\[8pt] 
  \displaystyle
 \  0 = Q_K \vec \left\{ n \hS_M \ - \ \sum\nolimits_{i=1}^n u_2(\hat{R}_{i,M}) (X_i-\hmu_M)(X_i-\hmu_M)^T \right\}, \\
 \end{cases}
\ee  
where, as before, $\hat{R}_{i,M} = (X_i-\hmu_M)^T \hK_M (X_i-\hmu_M)$.  

Suppose now that $X_1, \ldots, X_n$ represents a random sample from $E_p(\mu,S,g)$. As previously noted,
the $M$-estimator $\hS_n$ fulfils Assumption \ref{ass:1} and, under general conditions, also Assumption \ref{ass:2}, and 
so Lemma \ref{lem:ae} applies. Sufficient conditions for Assumption \ref{ass:2} to hold are given in Assumption \ref{ass:an} 
of the appendix. We also explicitly state the following condition.
\begin{assumption}[Conditions on $u_1$ and $u_2$] \label{ass:u}
The functions $u_1$ and $u_2$ are non-increasing, while the functions $\phi_1(s) = s u_1(s)$ and
$\phi_2(s) = s u_2(s)$ are non-decreasing.
\end{assumption}
It turns out  the plug-in approach based on a full $M$-estimate and the graphical $M$-estimate approach are 
asymptotically equivalent. 
\begin{theorem} \label{th:asymeq}
Let $X_1, \ldots, X_n$ be independent and identically $E_p(\mu,S,g)$ distributed with $S^{-1} \in \Sp(G)$. If the functions $u_1$, $u_2$ and $g$ are such that Assumptions \ref{ass:u} and \ref{ass:an} are satisfied, then 
$\sqrt{n}\{(\hmu_P,\vec\hS_P) - (\hmu_M,\vec\hS_M)\} \to 0$ in probability.
\end{theorem}

The interesting fact that the plug-in and the graphical $M$-estimator are asymptotically equivalent at elliptical distributions is favourable for the plug-in $M$-estimator. The unconstrained $M$-estimator is well studied, existence and uniqueness are guaranteed for data in sufficiently general position, and algorithms for its computation have been shown to converge in theory and proven to work sufficiently fast in practice. 

On the other hand, a thorough assessment of the properties of the graphical $M$-estimator including existence, uniqueness and finite-sample properties, is yet due and goes beyond the scope of this paper. Also, the graphical $M$-estimator is presumably harder to compute. It can be solved by a double-loop, IRS-type algorithm, as proposed, e.g., by \citet{Finegold2011}
for the maximum likelihood estimate based on the elliptical $t$-distribution, where each iteration consists of a complete IPS algorithm \citep[cf.][]{Speed1986}. 
The construction of a reliable single-loop algorithm is also an open research question.

Thus, altogether, one can recommend to use the plug-in estimator for moderate to large sample sizes. Simulations show, however, that the graphical $M$-estimator can be substantially more efficient at small samples. Furthermore, the graphical $M$-estimator is computable for fewer observations. The existence of the unconstrained $M$-estimate and thus the plug-in $M$-estimator requires at least $p+1$ data points in general position. More generally, any robust, affine equivariant estimator requires at least $p+1$ data points \citep{Tyler2010}. For decomposable models $G$, the sample size must only be as large as the largest clique of $G$ for the graphical $M$-estimate to be computable. It is to be expected that results concerning the existence of the Gaussian graphical maximum likelihood estimator \citep{Buhl1993, uhler:2012} for general graphs $G$ can be extended to graphical $M$-estimators. 

\section{A statistical application}
In Sections \ref{sec:aese} and \ref{sec:Mgest} we have studied the problem of estimating a positive definite scatter matrix subject to the condition that it contains zero-entries in the inverse at specific off-diagonal positions, which are given by a graph $G$. In this section we want to exemplify the benefit of these considerations for the statistical analysis. 
Let in the following $\hS_n$ be any affine equivariant, asymptotically normal scatter estimator, and $\hS_G$ a corresponding constrained estimate, i.e. either the plug-in estimate $\hS_G = h_G(\hS_n)$ or, if $\hS_n$ is the full $M$-estimate satisfying (\ref{eq:Mest}), 
the graphical $M$-estimate satisfying (\ref{eq:Mgest1}).
We have derived the asymptotic distribution of $\hS_G$, which allows to construct estimators and tests for any aspect of scatter within the covariance selection model $G$. 
An example is the deviance test (\ref{eq:gen.dev}), which tests for a smaller model $G_0$, i.e., if the true scatter matrix $S$ satisfies some further zero partial correlation restrictions, additional to the ones already given by $G$. By incorporating the knowledge about the dependence structure that is mediated through the graph $G$, one is able to obtain more efficient statistical methods. Depending on the true parameter values, the gain in asymptotic efficiency can be quite large, but also nil, as Example \ref{ex:chordless-p-cycle} below demonstrates.
In this context, the scale-free aspects of scatter, i.e., those that remain invariant under overall scale changes, which include all aspects of dependence, such as correlation, partial correlation, principal components, ratios of eigenvalues, etc, are of particular interest. Their asymptotic distribution further simplifies, since the second term in (\ref{eq:W_{V,G}2}), related to $\sigma_2$, vanishes, and also the correction factor $\eta$ from Lemma \ref{lem:ae} cancels. For details see \citet{Tyler1983}. Example \ref{ex:chordless-p-cycle} is such a case. 
\begin{example}[Chordless-$p$-cycle] \label{ex:chordless-p-cycle} \rm
Consider the situation of $p$ variables and the chordless-$p$-cycle as graph $G = (V,E)$, i.e., 
$E = \{ \{i,i+1\}, \{1,p\} \mid i = 1,\ldots,p-1 \}$, which is, except for the trivial case $p = 3$, a non-decomposable graph. For $p = 7$, it is depicted in Figure \ref{fig:1}.
\begin{figure}[t]
\centering
\includegraphics[scale=0.43]{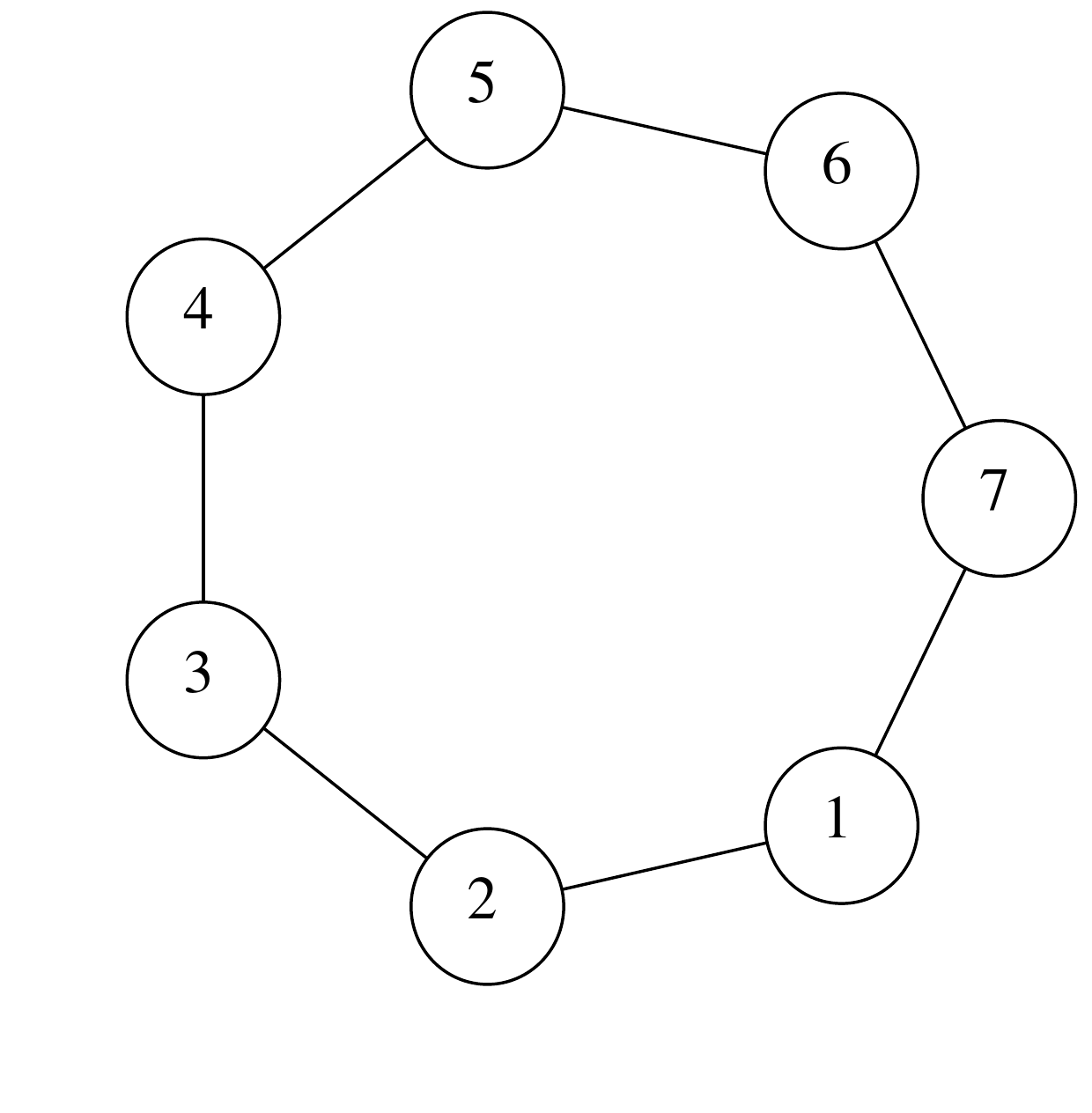}
\caption{Example graph: chordless-$7$-cycle} 
\label{fig:1}
\end{figure}
Assume that the data stem from a $p$-variate elliptical distribution $X \sim E_p(\mu,S,g)$. We fix a shape matrix $S$ fulfilling the graph $G$: All non-zero partial correlations have the same value $-1/2 < c < 1/2$, and the diagonal elements of $S$ are all equal, their specific not being of interest. This choice leads to a positive definite shape matrix $S$, which can be deduced from results about circulant matrices \citep[e.g.][]{gray:2006}. 
Assume further, we want to estimate the partial correlation $p_{1,2}$ between the first and second component of $X$ given all remaining components. Let 
\[
	\pi: \Ss^+_p \to \Ss_p: A \mapsto -\, A_D^{-1/2} A\, A_D^{-1/2},
\]
denote the function that maps the concentration matrix onto the corresponding matrix of pairwise partial correlations, cf.~\citet[][Chapter 5]{Whittaker1990}.
Here $A_D$ denotes the diagonal matrix that has the same diagonal as $A \in \Rpp$, and $A_D^{-1/2}$ is short for $(A_D)^{-1/2}$. With this notation, the parameter $p_{1,2}$ of interest can be written as
	$p_{1,2} = Q_{\{(2,1)\}} \vec\,\pi(S^{-1})$,
where, following the notational convention introduced at the beginning of Section \ref{sec:main}, the matrix $Q_{\{(2,1)\}}$ is of dimension $1 \times p^2$ and picks the second element of $\vec\, \pi(S^{-1})$. Let $\hS_n$ be a scatter estimator satisfying Assumptions \ref{ass:1} and \ref{ass:2}. We have two possible estimators for $p_{1,2}$ based upon $\hS_n$:  
the unconstrained estimator 
$\hat{p}_{1,2} = Q_{\{(2,1)\}} \vec\,\pi(\hK_n)$
and 
the graph-constrained estimator 
$\hat{p}_{1,2;G} = Q_{\{(2,1)\}} \vec\,\pi(\hK_G)$,
where $\hK_n = \hS_n^{-1}$ and $\hK_G = \{h_G(\hS_n)\}^{-1}$. The estimator $\hat{p}_{1,2;G}$ takes into account the information that $S^{-1} \in \Sp(G)$.
The derivative of $\pi$ is
\[
	\dsD \pi(A) = - M_p \left\{\pi(A) \otimes A_D^{-1}\right\}J_p \ - \ \left( A_D^{-1/2} \otimes A_D^{-1/2} \right) M_p,
\qquad A \in \Sp,
\]
where $J_p = \sum_{i=1}^p e_i^{} e_i^T \otimes e_i^{} e_i^T$, see the proof of Proposition 1 in \citet{Vogel2011}. Thus by means of the delta method we can compute from (\ref{eq:W_V}) and (\ref{eq:W_{u,G}}) the asymptotic variances of $\hat{p}_{1,2}$ and $\hat{p}_{1,2;G}$, respectively:
\[
	ASV(\hat{p}_{1,2}) = 2 \sigma_1 Q_{\{(2,1)\}} \dsD \pi(K) (S \otimes S)^{-1} \left\{\dsD \pi(K)\right\}^T Q_{\{(2,1)\}}^T,
\]
\[
	ASV(\hat{p}_{1,2;G}) = 2 \sigma_1 
		Q_{\{(2,1)\}} \dsD \pi(K) \Gamma	
		\left\{ \Gamma^T (S \otimes S) \Gamma \right\}^{-1}
		\Gamma^T \left\{\dsD \pi(K)\right\}^T Q_{\{(2,1)\}}^T,
\]
where $\Gamma = D_p \tQK^T$. The matrix $\Gamma$ serves as an inverse operator to $Q_K$, i.e., it maps $k = Q_K \vec K$ back to $\vec K$ such that $K$ is symmetric. 
The partial correlation is a scale-invariant property of the shape matrix $S$, the scalar $\sigma_2$ and the correction factor $\eta$ both vanish.
The asymptotic relative efficiency $ARE(\hat{p}_{1,2;G},\hat{p}_{1,2}) = ASV(\hat{p}_{1,2}) / ASV(\hat{p}_{1,2;G})$
of the constrained estimator $\hat{p}_{1,2;G}$ with respect to the unconstrained estimator $\hat{p}_{1,2}$ is always greater than or equal to 1. This asymptotic relative efficiency is the same for any pair of partial correlation estimates that are derived from the same scatter estimate $\hS_n$. Specific numbers for several values of $c$ and $p$ are given in Table \ref{tab:1}.
\end{example}
\begin{table}[t]
{\small 
\caption{Asymptotic relative efficiency of a graph-constrained partial correlation estimator with respect to the corresponding unconstrained estimator}
	\begin{tabular}{c|cccccccccccccc} 
		c	& \multicolumn{14}{c@{\ }}{ dimension $p$ } \\[.6ex]
	    & 4 &  5 &  6 &  7 &  8 &  9 & 10 & 11 & 12 & 13 &   20 &  30 & 50 \\[0.7ex]	
	    \hline
	$0$ &  $1.00$ &  $1.00$ &  $1.00$ &  $1.00$ &  $1.00$ &  $1.00$ &  $1.00$ &  $1.00$ &  $1.00$ &  $1.00$ &  $1.00$ &  $1.00$ &  $1.00$  \\
	$-0.05$ &  $1.01$ &  $1.01$ &  $1.01$ &  $1.01$ &  $1.01$ &  $1.01$ &  $1.01$ &  $1.01$ &  $1.01$ &    $1.01$ &   $1.01$ &   $1.01$ &  $1.01$ \\
	$-0.1$ &  $1.02$ &  $1.02$ &  $1.02$ &  $1.02$ &  $1.02$ &  $1.02$ &  $1.02$ &  $1.02$ &  $1.02$ &  $1.02$ &    $1.02$ &   $1.02$ &  $1.02$ \\
	$-0.2$ &  $1.08$ &  $1.09$ &  $1.09$ &    $1.09$ &  $1.09$ &  $1.09$ &  $1.09$ &  $1.09$ &  $1.09$ &  $1.09$ &    $1.09$ &    $1.09$ &  $1.09$ \\
	$-0.3$ &  $1.18$ &  $1.24$ &  $1.23$ &  $1.23$ &  $1.23$ &  $1.23$ &  $1.23$ &  $1.23$ &  $1.23$ &  $1.23$ &    $1.23$ &   $1.23$ &  $1.23$ \\
	$-0.4$ &  $1.32$ &  $1.55$ &  $1.49$ &  $1.54$ &  $1.52$ &  $1.54$ &  $1.53$ &  $1.53$ &  $1.53$ &  $1.53$ &    $1.53$ &   $1.53$ &  $1.53$ \\
	$-0.49$ &  $1.48$ &  $2.27$ &  $1.93$ &  $2.43$ &  $2.12$ &  $2.44$ &  $2.22$ &  $2.43$ &  $2.27$ &  $2.41$ &  $2.35$ &    $2.36$ &  $2.36$ \\
	\end{tabular} \label{tab:1}
}
\end{table}

\section{Discussion}
A covariance selection model is just one instance of a model where some further structure on the covariance matrix of multivariate data is assumed, which allows to work with fewer parameters.
When we want to robustly analyse such structured covariance models, as well as in many other situations, 
there are two basic approaches of constructing robust estimates:
One is to simply use a robust estimate instead of the usual, non-robust estimate, here the sample covariance matrix, and apply any subsequent analysis in an analogous manner. This is the plug-in approach. Often, estimates are defined as the optimizing point of some criterion function. An alternative approach is thus to alter the criterion function such that the influence of outlying observations is reduced. This approach is usually referred to as $M$-estimation. In the case of Gaussian graphical models we have that the maximum likelihood estimator $\hat{K}_G$ for the concentration matrix $K$ is the maximizing point of 
\begin{eqnarray}
	\varphi(K) & = & \log \det K  - \trace( K \hSigma_n ) \label{eq:phi.1}\\
	           & = & \log \det K  
	              - n^{-1} \sum\nolimits_{i=1}^n \log \gamma
	              \left\{ (X_i - \bar{X}_n)^T K (X_i - \bar{X}_n) \ \right\}
	               \label{eq:phi.2}
\end{eqnarray}
within the set $\Sp(G)$, where $\gamma(y) = \exp(y)$, $y \ge 0$. Representation (\ref{eq:phi.1}) immediately suggests the plug-in approach, whereas representation (\ref{eq:phi.2}) points to the $M$-approach. The results of the previous section indicate that for an appropriate choice of the replacements for $\hSigma_n$ and $\gamma$, both approaches are asymptotically equivalent. It is a very interesting research question to quantify under what conditions and for which structured covariance models this holds true.

\section*{Acknowledgement} 
The authors thank Roland Fried for very stimulating discussions initiating this research. The first author was supported in part by the German Research Foundation. 

\appendix

\section*{Appendix: Proofs}

Before proving Proposition \ref{prop:1} we have to introduce some more notation and, in particular, clarify what we understand as the derivative of a function that maps symmetric matrices to symmetric matrices. The function $v(\cdot)$, introduced at the beginning of Section \ref{sec:main}, is properly defined as
$v: \Ss_p \to \R^m: A \mapsto D_p^+ \vec A$ and its inverse as $v^{-1}: \R^m \to \Ss_p: a \mapsto \matpp D_p a$.
We use the following notational convention: For any $A \in \Rpp$ we write $\tilde{A}$ for $v(A)$. The use of \, $\tilde{}$ \, henceforth indicates an $m$-dimensional object. For functions $h:\Ss_p \to \Ss_p$, we write $\tilde{h}$ to denote the corresponding function mapping $v(A)$ to $v\{h(A)\}$. Furthermore, for any set $\Cc \subseteq \Ss_p$ let
\[
	\tilde\Cc	= \left\{ x \in \R^m \, \middle| \, v^{-1}(x) \in \Cc \right\}.
\]
Then $v$ is also a bijection from $\Sp$ to $\tSp$, and we will henceforth consider it restricted to this space. The set $\tSp$ is open in $\R^m$. We say that a function $h:\Cc \subseteq \Ss_p \to \Ss_p$ is continuously differentiable on $\Cc$ if $\tilde\Cc$ is open in $\R^m$ and $\tilde{h}:\tilde\Cc \to \R^m$ is continuously differentiable.
Letting $\dsD\tilde{h}(x)$ denote the Jacobi matrix or derivative of $\tilde{h}$ at point $x \in \R^m$, we define the derivative $\dsD h(A)$ of $h: \Cc \subseteq \Ss_p \to \Ss_p$ at point $A \in \Cc$ as the $p^2 \times p^2$ matrix given by
\be \label{eq:derivative2}
	\dsD h(A) \ = \ D_p \dsD \tilde{h} (\tilde{A}) D_p^+.
\ee
This definition is determined by the requirements 
(i) $\dsD \tilde{h} (\tilde{A}) \ = \ D_p^+ \dsD h(A) D_p$ \ and 
(ii) $K_p \dsD h(A) \ = \ \dsD h(A) K_p \ = \ \dsD h(A)$,
which, roughly speaking, say that 
(i) $\dsD h(A)$ ought to be an appropriate representation of $\dsD \tilde{h}(\tilde{A})$ and 
(ii) $\dsD h(A)$ should also reflect the symmetry that the argument as well as the value of $h$ possess. Thus, in order to show the differentiability of $h_G$, we will consider the function $\tilde{h}_G$ and compute its derivative, from which by (\ref{eq:derivative2}) the expression for $\dsD h_G$ given in (\ref{eq:derivative1}) readily follows.

We declare some further notation related to the graph $G$. Let
\[
	\tPG = 
	\left(
		\renewcommand{\arraystretch}{1.5}
		\begin{array}{c}
		\tQKG \\
		\hdashline
		\tQDG
		\end{array}
	\right)
	\quad \in \R^{m \times m}.
\]
The matrix $\tPG$ is orthogonal.
For $a \in \R^{m-q}$ and $b \in \R^q$ define
\[
	\langle a;b\rangle_G = \tPG^T 
	\begin{pmatrix}
		a \\
		b
	\end{pmatrix} \ \in \ \R^m, 
\]
i.e., the operation $\langle\cdot;\cdot\rangle_G$ fills an $m$-vector with the elements of $a$ and $b$ in such a way that $\tQKG \langle a;b\rangle_G = a$ and $\tQDG \langle a;b\rangle_G = b$. Let $(\tSp)_G$ be the set of all $(m-q)$-vectors for which there is a $y \in \R^q$ such that $\langle x;y\rangle_G \in \tSp$. The set $(\tSp)_G$ is open in $\R^{m-q}$.

Although, by going from $h_G$ to $\tilde{h}_G$, we have eliminated the redundancy due to the symmetry of the matrices, the function $\tilde{h}_G$ contains further redundancies. Recall the original definition of $h_G$, given by (\ref{eq:h_G}):
The function $h_G$ maps an unconstrained covariance estimate $\hat{\Sigma}_n$ to the corresponding constrained covariance estimate $\hat{\Sigma}_G$ under the model $G$. It takes $p(p+1)/2 - q$ values, $p$ estimated variances $\hat{\sigma}_{i,i}$, $1 \le i \le p$, and $p(p-1)/2 -q$ estimated covariances $\hat{\sigma}_{i,j}$, $\{i,j\} \in E$,
and produces $q$ new values: covariance estimates $\hat{\sigma}_{i,j}$ for $\{i,j\} \notin E$, $i \neq j$. So $h_G$, as well as $\tilde{h}_G$, are actually functions from $\R^{m-q}$ to $\R^q$. They may be further reduced to the function $t_G$, defined by
\[
	t_G: (\tSp)_G 
	\to \R^q: x \mapsto \tQDG \tilde{h}_G\left( \langle x ; y \rangle_G\right),
\]
where $y$ is some $q$-vector such that $\langle x;y\rangle_G \in \tSp$. Then $\tilde{h}_G$ can be expressed as
\be \label{eq:tilde{h}_G.vs.t_G}
	\tilde{h}_G(x) = \tPG^T 
	\begin{pmatrix}
	\tQKG x \\
	t_G\left\{ \tQKG x \right\}
	\end{pmatrix}, 
	\qquad x \in \R^m.
\ee
The function $t_G$, and thus $\tilde{h}_G$, is defined implicitly through the function
\[
	H_G: \tSp \subset \R^m \to \R^q: x \mapsto \tQDG v\left[ \left\{v^{-1}(x)\right\}^{-1}\right].
\]
The inner $^{-1}$ refers to the inverse function of $v$, whereas the outer $^{-1}$ refers to matrix inversion.
For any $x \in (\tSp)_G$, the value $y = t_G(x)$ is the unique solution to 
\be \label{eq:h_G:2}
	H_G\left(\langle x; y \rangle_G \right) \ = \ 0.
\ee
This is a reformulation of (\ref{eq:h_G}), and from the theory of Gaussian graphical models we know that $t_G$ is well defined, i.e.\ that for every $x \in (\tSp)_G$, there is indeed exactly one solution to (\ref{eq:h_G:2}). We are now ready to prove Proposition \ref{prop:1}.
\begin{proof}[Proof of Proposition \ref{prop:1}]
Part (I): 
We prove that $\tilde{h}_G$ is continuously differentiable by means of the implicit function theorem.
Let $x \in (\tSp)_G$ be fixed. There exists a unique $y \in \R^q$ such that 
$H_G\left(\langle x; y \rangle_G \right) = 0$ and $\langle x; y \rangle_G \in \tSp$. The Jacobi matrix of 
$H_G\left(\langle x ;  \cdot \rangle_G \right)$,
i.e.\ the matrix of all partial derivatives of $H_G\left(\langle x ; y \rangle_G \right)$ with respect to $y$, is  
\be \label{eq:derivative3}
	d H_G\left(\langle x ; y \rangle_G \right) / d y  
	\ = \ - Q_{D(G)} (A^{-1} \otimes A^{-1}) D_p \tQDG^T,
\ee
where $A = v^{-1}\left(\langle x ; y \rangle_G \right)$.
Due to the assumption $\langle x; y \rangle_G \in \tSp$, (\ref{eq:derivative3}) is an invertible matrix.
By the implicit function theorem \citep[e.g.][Theorem 6.4.1]{Trench2003}, there exists a continuously differentiable function $t_x : U_x \to \R^q$,
defined on some open neighbourhood $U_x$ of $x$ with $U_x \subset (\tSp)_G$, such that $t_x(x) = y$ and 
$H_G\left\{\langle z ; t_x(z) \rangle_G \right\} = 0$ for all $z \in U_x$. Since $t_G$ is the unique function defined on $(\tSp)_G$ that satisfies 
\be \label{eq:deriv3}
	H_G\left\{ \langle z; t_G(z) \rangle_G\right\} = 0 
\ee
for all $z \in (\tSp)_G$, we have $t_x = t_G|_{U_x}$. This holds true for every $x \in (\tSp)_G$, hence $t_G$, and by (\ref{eq:tilde{h}_G.vs.t_G}) also $\tilde{h}_G$, is continuously differentiable.
\par
Part (II): We use implicit differentiation, see e.g.\ \citet[][Theorem 6.4.1]{Trench2003}. Differentiating both sides of (\ref{eq:deriv3}) with respect to $z$ yields
\[
	\dsD t_G(x) =
	- \left[  \frac{d H_G\left\{\langle x ; t_G(x) \rangle_G\right\}}{d y} \right]^{-1}
						\frac{d H_G\left\{\langle x ; t_G(x) \rangle_G\right\}}{d x},
	\qquad x \in (\tSp)_G, 
\]
where $d H_G\left\{\langle x ; t_G(x) \rangle_G\right\} / d y$ denotes the derivative of $H_G\left(\langle x ; \cdot \rangle_G\right)$ evaluated at the point $t_G(x) \in \R^q$. With the notation introduced above we have
\[
	\dsD t_G(x) =
	- \left\{ 	Q_{D(G)} (A_G^{-1} \otimes A_G^{-1}) D_p \tQDG^T \right\}^{-1}
						Q_{D(G)} (A_G^{-1} \otimes A_G^{-1}) D_p \tQKG^T,
	\quad \ x \in (\tSp)_G,
\]
where $A_G = v^{-1} \left\{\langle x; t_G(x) \rangle_G \right\}$. By (\ref{eq:tilde{h}_G.vs.t_G}) we find further
\[
	\dsD \tilde{h}_G(x)
	\ = \ 
	\tPG^T \,  
	\left(
		\renewcommand{\arraystretch}{1.5}
		\begin{array}{c}
		I_{m-q} \\
		\hdashline
		\dsD t_G\left(\tQKG x \right)
		\end{array}
	\right) \,
	\tQKG 
	\]
	\[
	\ = \ 
	\tQKG^T \tQKG \ - \ 
	\tQDG^T \left\{ 	Q_{D(G)} (A_G^{-1} \otimes A_G^{-1}) D_p \tQDG^T \right\}^{-1}
							\!	Q_{D(G)} (A_G^{-1} \otimes A_G^{-1}) D_p \tQKG^T \tQKG
\]
for any $x \in \tSp \subset \R^m$, where now $A_G$ denotes $v^{-1}\left\{ h_G(x) \right\}$.
With (\ref{eq:derivative2}) and noting that $D_p \tQDG^T = 2 M_p Q_{D(G)}^T$, we arrive at
\be \label{eq:Dh_G}
	\dsD h_G(A) \ = \ 
	M_{p,G}  -  
	M_p Q_{D(G)}^T \left\{ 	Q_{D(G)} (A_G^{-1} \otimes A_G^{-1}) M_p Q_{D(G)}^T \right\}^{-1}
										\!	Q_{D(G)} (A_G^{-1} \otimes A_G^{-1}) M_{p,G}
\ee
for $A \in \Sp$, where $A_G = h_G(A)$ and $M_{p,G} = D_p \tQKG^T \tQKG D_p^+$. The matrix $M_{p,G}$ is obtained from $M_p$ by putting all rows and columns that correspond to non-edge positions of $G$, sub-diagonal as well as super-diagonal, to zero. Noting that $M_p - M_{p,G} = 2 M_p Q_{D(G)}^T Q_{D(G)} M_p$, we find that $M_{p,G}$ may be replaced by $M_p$ in (\ref{eq:Dh_G}), and we obtain the expression given in (\ref{eq:derivative1}). This completes the proof of Proposition \ref{prop:1}.
\end{proof}
The general method of proof applied here is also described in \citet{Benichou1989}.
\begin{proof}[Proof of Theorem \ref{th:1}]
Part (\ref{th:1(I)}) is a version of the delta method, parts (\ref{th:1(II)}-\ref{th:1(IV)}) follow by straightforward matrix calculus. For part (\ref{th:1(IV)}), one obtains by the delta method directly from (\ref{eq:W_{V,G}2})
\[
	W_{u,G} = 2 \sigma_1 \left\{ 
		\tQK \Omega \tQK^T - \tQK \Omega \tQD^T \left(\tQD \Omega \tQD^T\right)^{-1} \tQD \Omega \tQK^T	
	\right\} + \sigma_2 u u^T,
\]
where $\Omega = D_p^+ (U \otimes U)(D_p^+)^T$, $U = V^{-1}$ and $u = Q_K \vec U = \tQK D_p^+ \vec U$. By the 
formula for the inverse of a partitioned matrix one identifies the matrix inside $\{ \}$ as the inverse of 
$\tQK \Omega^{-1} \tQK^T =$ \ $\tQK D_p^T (V\otimes V) D_p \tQK^T$.
\end{proof}
\begin{proof}[Proof of Lemma \ref{lem:ae}]
Part (\ref{lem:ae(I)}): Consider the special case $\mu = 0$, $S = I_p$ and let $X_1,\ldots,X_n$ be independent and identically $E_p(0,I_p,g)$ distributed. Then, for any orthogonal $O \in \Rpp$, let $\X_n^* = \X_n O^T$. We have $\X \sim \X_n^*$ and by Assumption \ref{ass:1} also $\hS_n(\X_n) \sim O \hS_n(\X_n) O^T$, where $\sim$ denotes equality in distribution. Assumption \ref{ass:2} implies $\hS_n(\X_n) \to V$ in probability. Hence by the continuous mapping theorem, $V = O V O^T$ for all orthogonal matrices $O$, hence $V = \eta I_p$ for some $\eta \ge 0$. The result for general $S$ follows again by the affine equivariance of $\hS_n$ and the continuous mapping theorem. We may restrict $\eta$ to positive values, since $S$ and $V$ are assumed to be positive definite. 

Part (\ref{lem:ae(II)}): Let $X_1,\ldots,X_n$ be $E_p(\mu,S,g)$ distributed. Let $O \in \Rpp$ again be orthogonal and $T = S^{1/2}O S^{-1/2}$. Due to the ellipticity we have $\X_n \sim (\X_n - 1_n \mu^T) T^T + 1_n\mu^T$, and by Assumption \ref{ass:1} also $ \sqrt{n} T\{\hS_n(\X_n)-\eta S\} T^T  \sim \sqrt{n}\{\hS_n(\X_n)-\eta S\}$, where, as before, $\sim$ denotes equality in distribution. By Assumption \ref{ass:2} and the continuous mapping theorem, we find that $Z$ fulfils the invariance property described in Remark 1 (I). The form (\ref{eq:W_V}) of the covariance matrix follows with \citet[][Corollary 1]{Tyler1982}.
\end{proof}

Derivation of (\ref{eq:Mgest1}) for the case $u_1 = u_2 = u$. \ 
Let $L_o(\mu,K)$ denote the criterion function in (\ref{eq:gemle}) and let $S = K^{-1}$. With $R_i = (X_i-\mu)^T K (X_i-\mu)$, we have
$\partial R_i/\partial \mu = - 2 (X_i-\mu)^T K$ 
and  
$\partial R_i/\partial K_{j,k} = (2-\delta_{j,k})e_j^T(X_i-\mu)(X_i-\mu)^Te_k$,
where $K_{j,k}$ are the elements of $K$, and $\delta_{j,k} = 0$ or $1$ for $j \ne k$ and $j = k$ respectively. 
Also,
$\partial \log\{\det(K)\}/\partial K_{j,k} = (2-\delta_{j,k}) e_j^T S e_k$. 
Thus,
$\partial L_o(\mu,K)/\partial \mu = 2\sum_{i=1}^n u(R_i) \,(X_i-\mu)^T K$, and
$\partial L_o(\mu,K)/\partial K_{j,k} = (2-\delta_{i,j})e_j^T\, \sum_{i=1}^n \big\{-u(R_i) (X_i - \mu)(X_i-\mu)^T + S\big\} e_k$ for $(j,k) \in K(G)$.
Setting these partial derivatives to zero gives (\ref{eq:Mgest1}) with $u_1 = u_2 = u$. 
\hfill $\square$ 

\medskip
Before giving the proof for Theorem \ref{th:asymeq}, we review some general results for $M$-estimating equations. 
Let $X_1, \ldots, X_n$ be a sample in 
an open subset of $\R^p$. 
An M-estimate for a parameter $\theta \in \Theta$, with $\Theta$ being an open subset of $\R^l$, can be defined as a solution $\hat{\theta}$ to the
$M$-estimating equations given by
\be \label{eq:Mequation}
	\ave \left\{ \psi_j (X_i;\hat{\theta})  \right\} \  = \ 0,  \qquad j = 1, \ldots, l,
\ee
where the average, here as well as in all following occurrences, is taken over $i = 1,\ldots,n$.
When $X_1, \ldots, X_n$ represent a random sample, the asymptotic normality of $M$-estimates is known to hold under very general conditions on the function $\psi = (\psi_1, \ldots, \psi_l)$ and on the underlying distribution $F$. We refer the reader to \citet{Huber2009}, \citet{Hampel1986} or \citet{Maronna2006} for further details. 
Central to the proof of asymptotic normality of an $M$-estimate, and central to our proof of Theorem \ref{th:asymeq}, is the expansion of $\psi(x;\hat{\theta})$ about the population value $\theta_F$. Rather than re-state the somewhat technical conditions needed for the aforementioned expansion to be applicable, we simply assume the following condition holds. 
For convenience, we use the notation $\partial f(x,y_o)/\partial y = \{\partial f(x,y)/\partial y\}|_{y=y_o}$.
\begin{assumption}[M-estimation regularity conditions] \label{ass:an}
The function $\psi$ and the distribution $F$ satisfy sufficient regularity conditions to ensure: 
\begin{enumerate}[(I)]
\item \label{enum:an:1}
There is a unique solution, $\theta_F = T(F)$, to the M-functional equation $E_F \{\psi(X;\theta)\} = 0$. 
\item \label{enum:an:2}
For any sequence $\hat{\theta}$ satisfying (\ref{eq:Mequation}), $\hat{\theta}\to \theta_F$ in probability.  
\item \label{enum:an:3}
For $j = 1, \ldots, l$, we have 
\[
	0 = \ave\{\psi_j(X_i;\theta_F)\} + \sum\nolimits_{k=1}^l (\hat{\theta}_{k} - \theta_k) \left[\ave\{\partial \psi_j(X_i;\theta_F)/\partial \theta_k  \}+ r_{j,k,n}\right],
\] 
where $\theta_F = (\theta_1,\ldots,\theta_l)$ and $r_{j,k,n} \to 0$ in probability as $n \to \infty$. 
\item \label{enum:an:4}
The expectations $b_{j,k} = E_F\{\partial \psi_j(X_i;\theta_F)/ \partial \theta_k\}$ and 
$a_{j,k} = E_F\{\psi_j(X_i;\theta_F)\psi_k(X_i;\theta_F)\}$ exist.
\end{enumerate}
\end{assumption}  
\begin{remark} 
Assumption \ref{ass:an} ensures that $B\{ \sqrt{n}(\hat{\theta} - \theta_F)\} \to N_l(0,A)$ in distribution, where the elements of $A$ and $B$
are $a_{j,k}$ and $b_{j,k}$ respectively. Furthermore, if $B$ is non-singular, then $\sqrt{n}(\hat{\theta} - \theta_F)$ converges
in distribution to a multivariate normal with mean zero and variance-covariance matrix $B^{-1}A(B^T)^{-1}$.
Also, a sufficient condition for Assumption \ref{ass:an} (\ref{enum:an:3}) to hold is that Assumption \ref{ass:an} (\ref{enum:an:2}) holds and $\psi$ has bounded second derivatives. 
\end{remark}
\begin{proof}[Proof of Theorem \ref{th:asymeq}] 
Since $(\hS_P)_{j,k} = (\hS_n)_{j,k}$ for $(j,k) \in K(G)$ and $(\hS_P^{-1})_{j,k} = (\hS_M^{-1})_{j,k} = 0$ for $(j,k) \in D(G)$, 
it is sufficient to show that $\sqrt{n}\{\hmu -\hmu_G\} \to 0$ and
\mbox{$\sqrt{n}\{(\hS_n)_{j,k} -(\hS_M)_{j,k}\} \to 0$} in probability for $(j,k) \in K(G)$. 

Assumption \ref{ass:an} (\ref{enum:an:1}) and (\ref{enum:an:2}) implies that $(\hmu_n,\hS_n)$ and $(\hmu_M,\hS_M)$ converge in probability to the same value, 
namely $(\mu,V)$, with $V =\gamma S$ and $\gamma$ being the unique solution to the equation 
\begin{equation} \label{s}
E\{\phi_2(R)\} = p, \quad \mbox{where} \quad  R = Z^TZ/\gamma, \qquad Z \sim E_p(0,I_p,g). 
\end{equation}
Rather than finding the partial derivatives in Assumption \ref{ass:an} (\ref{enum:an:3}) explicitly, it is easier to use perturbation techniques to obtain the linear expansions.  Consequently, we have for the full $M$-estimate
\begin{equation} \label{mulinear}
		0 = \ave\{u_1(R_i)(X_i-\mu)\} - \{B_{1,n} + o_p(1) \} ( \hmu_n-\mu ), 
\end{equation}
\begin{equation} \label{Vlinear}
		0 = \vec
		\left[\ave\{u_2(R_i)(X_i-\mu)(X_i-\mu)^T\} - V \right] 
		- \{ B_{2,n} + o_p(1) \} \vec\,(\hS_n - V), 
\end{equation}
where 
$R_i = (X_i -\mu)^T V^{-1} (X_i - \mu)$, 
$B_{1,n} = \ave\{u_1(R_i)\} I_p \, + \, 2\ave\{u_1'(R_i)(X_i-\mu)(X_i-\mu)^T\}V^{-1}$ and 
\mbox{$ B_{2,n} = \ave\{u_2'(R_i) (X_i-\mu)(X_i-\mu)^T \otimes (X_i-\mu)(X_i-\mu)^T \} (V^{-1} \otimes V^{-1}) + I_{p^2}$.} 
Likewise, the linear expansions
for (\ref{eq:Mgest1}) are 
\begin{equation} \label{muglinear}
	  0 = \ave\{u_1(R_i)(X-\mu)\} - \{B_{1,n} + o_p(1) \} (\hmu_M-\mu), 
\end{equation}
\begin{equation} \label{Vglinear}
	0 = (e_k^T \otimes e_j^T) 
	\left(
	   \vec \left[\ave\{u_2(R_i)(X-\mu)(X-\mu)^T\} - V \right] 
	    - \{ B_{2,n} + o_p(1)\} \vec\,\{\hS_M - V\}
	\right)
\end{equation}
for $(j,k) \in K(G)$.

Consider the location component. By the law of large numbers, it follows that \mbox{$B_{1,n} \to B_1$} in probability, where
$B_1 = E\{u_1(R)\} I_p + (2/\gamma) ~S^{1/2}E\{u_1'(R)ZZ^T\}S^{-1/2}$, 
with $R$ and $Z$ defined in (\ref{s}). Assumption \ref{ass:an} (\ref{enum:an:4}) assures that $B_1$ exists. 
Evaluating the expectations gives $B_1 = b_1 I_p$, where
$b_1 = E\{u_1(R)\} + 2~E\{R\,u_1'(R)\}/p = (1-2/p)E\{u_1(R)\} + (2/p) E\{\phi_1'(R)\}.$
By Assumption \ref{ass:u}, $b_1 > 0$ and hence $B_1$ is non-singular. This implies that $\sqrt{n}(\hmu_n - \mu)$ and $\sqrt{n}(\hmu_M - \mu)$
converge in distribution to a multivariate normal distributions, and so $\sqrt{n}\{\hmu_n - \hmu_M \}= O_p(1)$. Subtracting (\ref{mulinear})
from (\ref{muglinear}) and multiplying by $\sqrt{n}$ then yields $0 = B_{1,n}\sqrt{n}\{\hmu_n - \hmu_M\} + o_p(1)O_p(1)$. Since,
$B_{1,n} \to b_1 I_p$ in probability, it follows that $\sqrt{n}(\hmu_n - \hmu_M) \to 0$ in probability.

For the scatter component, we again have by the law of large numbers that $B_{2,n} \to B_2$ in probability, where
\[
	B_2 = \gamma^{-2}(S^{1/2} \otimes S^{1/2}) E\{u_2'(R) (ZZ^T \otimes ZZ^T) \} (S^{-1/2} \otimes 	S^{-1/2}) + I_{p^2},
\]
with Assumption \ref{ass:an} (\ref{enum:an:4}) assuring that $B_2$ exists. Evaluating the expectation gives
$B_2 = (S^{1/2} \otimes S^{1/2})B_o(S^{-1/2} \otimes S^{-1/2})$, 
where 
\[
	B_o = (1+ b_2)I_{p^2} + b_2 K_p + b_2 \vec\,(I_p)\vec\,(I_p)^T, 
	\qquad 
	b_2 = E\{R^2u_2'(R)\}/\{p(p+2)\},
\]
thus \mbox{$B_2 = (1+b_2)I_{p^2} + b_2 K_p + b_2 \vec\,(S)\vec\,(S^{-1})^T$.}
The $p^2$ eigenvalues of $B_o$ are $1+2b_2$ repeated $p(p+1)/2 -1$ times, $1$ repeated $p(p-1)/2$ times and $1+(p+2)b_2$, which occurs once.  
Since, by Assumption \ref{ass:u}, $b_2 < 0$, it follows that $\lambda = 1+(p+2)b_2$ is the smallest eigenvalue of $B_o$. 
Since $s^2u_2'(s) = s\phi_2'(s) - \phi_2(s)$ and $E\{\phi_2(s)\} = p$, we have by Assumption \ref{ass:u} that $\lambda = E\{s\phi_2'(s)\}/p > 0$. Hence, $B_o$ and consequently $B_2$, is non-singular. 
 This implies that 
$\sqrt{n} \{(\hS_n)_{j,k} - (\hS_M)_{j,k}\} = O_p(1)$  for $(j,k) \in K(G)$. 
Pre-multiplying (\ref{Vlinear}) by $e_j^T \otimes e_k^T$, subtracting it from (\ref{Vglinear}), and then multiplying by
$\sqrt{n}$ gives 
$0 = (e_j^T \otimes e_k^T)B_{2,n}\sqrt{n} \vec\,\{\hS_n -\hS_M\} + o_p(1)O_p(1)$, 
and so 
$(e_j^T \otimes e_k^T)B_2 \sqrt{n} \vec\,(\hS_n -\hS_M)  \to 0$ in probability for $(j,k) \in K(G)$. 
This last limit can be expressed as
\begin{equation} \label{Vlimit}
 (1+2b_2)\sqrt{n}\{(\hS_n)_{j,k} - (\hS_M)_{j,k}\} 
 	+ S\!_{j,k} \, b_2  \trace\{ S^{-1} \sqrt{n} (\hS_n -\hS_M)\}
 	\ \to \ 0,
 	\qquad (j,k) \in K(G)
\end{equation} 
in probability. Now 
\[
	\trace\{ S^{-1} (\hS_n -\hS_M) \} 
	\ = \ \sum_{(r,c) \in K(G)} (2-\delta_{r,c})(S^{-1})_{r,c} \{(\hS_n)_{r,c} - (\hS_M)_{r,c}\},
\]
since  
$(S^{-1})_{r,c} \{(\hS_n)_{r,c} - (\hS_M)_{r,c}\} = (S^{-1})_{c,r} \{(\hS_n)_{c,r} - (\hS_M)_{c,r}\}$
due to the symmetry and
$(S^{-1})_{r,c} = 0$ for $(r,c) \in D(G)$. 
Recall the elements of $\sqrt{n}\{(\hS_n)_{j,k} - (\hS_M)_{j,k}\}$ for
$(j,k) \in K(G)$ 
can be represented by 
$Y_n = \sqrt{n} Q_K  \vec\,(\hS_n -\hS_M)$. Thus (\ref{Vlimit}) can be expressed as $CY_n \to 0$ in probability for a $(m-q)\times(m-q)$ matrix $C$, the elements of which are specified below. 
For a matrix position $(j,k) \in K(G)$ of some $p\times p$ matrix $D$, say, we let $\tau(j,k) \in \{1,\ldots,m-q\}$ denote the position of the element $D_{j,k}$ in the vector $Q_K \vec D$. The number $\tau(j,k)$ is the rank of $(j,k)$ when ordering the elements of $K(G)$ according to the the ordering $\prec_p$, introduced at the beginning of Section \ref{sec:main}. Then we have for the diagonal elements of $C$, 
\[
	C_{\tau(j,k),\tau(j,k)} \ = \  
	1 + 2 b_2 + b_2 (2-\delta_{j,k})\, S\!_{j,k}\,(S^{-1})_{j,k}, \qquad (j,k) \in K(G),
\]
and for the off-diagonal elements
\[
	C_{\tau(j,k),\tau(r,c)} \ = \  
 	S\!_{j,k} \, b_2 (2-\delta_{r,c})(S^{-1})_{r,c}, 
 	\qquad (j,k), (r,c) \in K(G), (j,k) \neq (r,c). 
\]
The matrix $C$ can be shown to be non-singular and so
$Y_n \to 0$ or equivalently $\sqrt{n}\{(\hS_n)_{j,k} - (\hS_M)_{j,k}\}  \to 0$ in probability for $(j,k) \in K(G)$. This completes the proof of Theorem \ref{th:asymeq}.
\end{proof}

\bibliographystyle{plainnat}

\end{document}